\documentclass[11pt,reqno]{amsart}

\usepackage{amssymb}
\usepackage{amsmath}
\usepackage{amsthm}
\usepackage{amsfonts}
\usepackage{bbm}
\usepackage{esint}
\usepackage{bm}
\usepackage[dvipsnames]{xcolor}
\usepackage{a4wide}

\usepackage{bookmark}
\usepackage{hyperref}
\hypersetup{pdfstartview={FitH}}

\newtheorem{theorem}{Theorem}
\newtheorem{lemma}[theorem]{Lemma}

\newcommand{\R}{\mathbb{R}}

\newcommand{\N}{\mathbbm{N}}
\renewcommand{\H}{\mathcal{H}}

\numberwithin{equation}{section}

\usepackage{cancel}

\begin{document}

\title{Quantitative bounds for products of simplices in subsets of the unit cube}

\author[P. Durcik]{Polona Durcik}
\address{Schmid College of Science and Technology, Chapman University, One University Drive, Orange, CA 92866, USA}
\email{durcik@chapman.edu}

\author[M. Stip\v{c}i\'{c}]{Mario Stip\v{c}i\'{c}}
\address{Schmid College of Science and Technology, Chapman University, One University Drive, Orange, CA 92866, USA}
\email{stipcic@chapman.edu}



\begin{abstract}
For each $1\leq i \le n$,  let $k_i\geq 1$ and let $\Delta_i$ be a set of vertices of a non-degenerate simplex of $k_i+1$ points in $\R^{k_i+1}$. If $A\subseteq [0,1]^{k_1+1}\times \cdots \times [0,1]^{k_n+1}$   is a Lebesgue measurable set of measure at least $\delta$, we show that there exists an interval $I=I(\Delta_1,\ldots, \Delta_n,A)$ of length at least $\exp(-\delta^{-C(\Delta_1,\ldots, \Delta_n)})$ such that for each $\lambda\in I$,   the set $A$ contains   $\Delta'_1\times \cdots \times \Delta'_n$, where each $\Delta_i'$ is an isometric copy of  $\lambda\Delta_i$.  This is a quantitative improvement of a result by Lyall and Magyar. Our proof relies on harmonic analysis. 
The main ingredient in the proof are cancellation estimates for forms similar to  multilinear singular integrals associated with $n$-partite $n$-regular hypergraphs.  
\end{abstract}
\maketitle


\section{Introduction}
Finite point configurations  in     large  but otherwise arbitrary  subsets  of  Euclidean spaces have been a topic of several recent investigations  \cite{Bou86, CMP15,  DK20, DK18, FKW, Kov20, LM20, LM18, LM19}.  
In this paper we study existence of isometric copies of   dilates    of products of non-degenerate simplices in sets of positive Lebesgue measure in the unit cube. 
The main focus   is on a  quantitative lower bound on the size of the family of dilates whose isometric copies are detected in the set, as stated in  the following  theorem.  
\begin{theorem}\label{mainthm}
Let $n\ge 1$. For each $1\le i \le n$, let  $k_i\ge 1$  and let    $$\Delta_i=\{0,v_i^1,\ldots, v_i^{k_i}\},$$ where $v_i^1,\ldots, v_i^{k_i}\in \R^{k_i}$ are linearly independent vectors. Let $\delta\in (0,1/2]$.  There exists a constant $C=C(\Delta_1,\ldots, \Delta_n)>0$  such that the following holds. For every measurable set   $A\subseteq [0,1]^{k_1+1}\times \cdots \times [0,1]^{k_n+1}$ with the Lebesgue measure $|A| \ge  \delta$, there is an interval $I=I(\Delta_1,\ldots, \Delta_n,A)\subseteq (0,\infty)$ of length at least 
  \[(\exp(\delta^{-C(\Delta_1,\ldots, \Delta_n)}) )^{-1} \] such that for every $\lambda\in I$ and $1\leq i \leq n$ one can find  $x_i,y_i^1,\ldots,y_i^{k_i}\in \R^{k_i+1}$ satisfying
  \[y_i^l\cdot y_i^m = \lambda v_i^l\cdot \lambda v_i^m\quad \textup{for} \quad 1\leq l,m \leq k_i\]
  and
  $\Delta'_1\times \cdots\times \Delta'_n \subseteq A$,  
  where 
  $\Delta'_i = \{x_i, x_i+y_i^1,\ldots, x_i+y_i^{k_i}\}$.
  That is, 
    $\Delta'_i$ is an isometric copy of $\lambda \Delta_i$. 
\end{theorem}
The case of a single simplex, i.e.\@ case  $n=1$ of   Theorem \ref{mainthm}, is a corollary of a stronger quantitative finitary result by Bourgain  \cite[Proposition 3]{Bou86}. The exact  quantitative bound on the length of the interval is not written explicitly in \cite{Bou86}, but it can be tracked down from its proof in \cite[Section 3]{Bou86}.
Lyall and  Magyar \cite{LM19} proved a version of Theorem \ref{mainthm} for an arbitrary finite product of simplices  but
with  the lower bound on the length of the interval $I$ being a reciprocal of a tower of exponentials 
  \[(\exp(\exp(\cdots \exp(\delta^{-C(\Delta_1,\ldots, \Delta_n)})) \cdots )) )^{-1} \] 
of length $n$. Their proof is based on a weak hypergraph regularity lemma developed in the setting of  Euclidean spaces.   Kova\v{c} and the first author \cite{DK20}  proved    Theorem \ref{mainthm} when $k_i=1$ for all $1\leq i\leq n$, in which case  each simplex has only two points, i.e.\@ their product is the vertex set of an $n$-dimensional rectangular box.  
Theorem \ref{mainthm} is thus an extension of \cite{DK20}   and a     quantitative improvement of \cite{LM19}. 

In analogy with the papers \cite{Bou86} and \cite{LM19} we  formulated Theorem \ref{mainthm} in the lowest dimension possible, i.e. for sets $A\subseteq [0,1]^{d_1}\times \cdots \times [0,1]^{d_n}$ where $d_i=k_i+1$. As in these papers,  our result extends to all higher dimensions $d_i \geq k_i+1$ with mainly notational changes throughout the proof. 
 
  Our approach builds on the methods from \cite{CMP15} and \cite{DK20}, which detect simplices through certain counting forms.  Following \cite{CMP15} and \cite{DK20}, in Section \ref{sec:forms}   we split the counting forms into their structured,  error, and uniform parts and analyze each of the parts in the subsequent sections. The error part is handled through a certain pigeonholing argument.  The main novelty of this paper is that  to estimate the error part, we need to show   cancellation in    multilinear singular integral forms associated with $n$-partite $n$-regular hypergraphs.  
  
Certain  dyadic models of the singular integral forms that we consider in this paper are known to be bounded in an open range of Lebesgue spaces. This was shown  by Kova\v{c}  \cite{Kov11}, who proved $L^p$ estimates for multilinear forms associated with graphs, and the second author  \cite{Sti19}, who extended this result to  hypergraphs. On the other hand, boundedness of the  continuous variants of the forms in \cite{Kov11} and their modified versions, which appear  in this paper, is  still  unknown in general. The only currently known bounds are in the special case of hypercubes, which are hypergraphs having two vertices in each partition class. This case  was investigated by    Slavikova, Thiele, and the first author \cite{DST}, who proved  estimates for the associated forms in an open range of Lebesgue exponents. 
However, to  establish Theorem \ref{mainthm}, one does not need to show $L^p$ boundedness of the  forms that arise in the proof  since a  weaker statement    suffices to deduce our combinatorial result. Namely,  it is enough to   establish only non-trivial quantitative cancellation bounds for the corresponding singular integral forms. That is, if the singular integral kernel  consists of finitely many consecutive scales, it suffices to   show a quantitative improvement over the trivial bound in the number of scales.  
  Moreover, to prove Theorem \ref{mainthm} it is enough to assume that all input functions are supported in the unit cube and bounded by $1$. 
The proof of our  cancellation  estimates uses an  extension of  the methods used to estimate singular Brascamp-Lieb type integrals as in \cite{DK20, DST, Kov11}, to the hypergraph setting.

A connection between cancellation estimates and  Euclidean Ramsey theorems was made explicit in   \cite{DK20}, which deals with arithmetic progressions in the unit cube. The paper \cite{DK20} shows that if $n\geq 3$, $p\in [1,\infty)\setminus\{1,2,\ldots, n-1\}$ and the dimension $d$ is large enough, then there exists a constant $C=C(n,p,d)$ such that the following holds. If $\delta \in (0,1/2]$ and    $A\subseteq [0,1]^d$ has Lebesgue measure at least $\delta$, then there is an interval $I$ whose length depends only on $n, p, d$, and $\delta$, such that for all $\lambda \in I$, the set $A$ contains  arithmetic progression $x,x+y,\ldots, x+(n-1)y\in A$ with $\|y\|_{\ell^p}=\lambda$. Here, the length of $I$ is a reciprocal of the double exponential of a negative power of $\delta$ when $n=3$ or  $4$ and the triple exponential when $n\geq 5$, which reflects the best known bounds in Szemer{\'e}di's theorem. The case $n=3$ also follows along the lines of the proof by Cook, Magyar, and Pramanik \cite{CMP15}.  However, we point out that in contrast with Theorem \ref{mainthm}, the   size of the gap $y$ in  \cite{DK20}  cannot be measured in all $\ell^p$ norms, in particular not in $\ell^2$,  as generalizations of Bourgain's counterexample \cite{Bou86} show. Moreover,  the result in \cite{DK20} is at present only known in sufficiently large dimensions.

In \cite{DK20},  bounds for the corresponding error part follow from cancellation estimates for singular integrals similar to the multilinear Hilbert transform. While non-trivial cancellation for the truncations of the  multilinear Hilbert transform to finitely many consecutive scales was first  shown   by Tao \cite{Tao15} and later also by Zorin-Kranich \cite{Zor15}, and Kova\v{c}, Thiele and the first author \cite{DKT16},  its boundedness on Lebesgue spaces
 is still one of the major open problems in harmonic analysis. The only known bounds exist for the bilinear Hilbert transform due to Lacey and Thiele \cite{LT}. 
In contrast to this, we expect that  boundedness of continuous variants of  multilinear singular integrals associated with hypergraphs, which arise in the proof of Theorem \ref{mainthm}, is within reach of current techniques. This is a problem for future investigation. In view of this,   we aim to keep the estimates for the error part and consequently also the proof of Theorem \ref{mainthm} as short as possible, i.e.  we   do not aim to optimize the bound in the number of scales for the singular integral as it would not improve the quantitative bound  in Theorem \ref{mainthm}. 

There are related questions on subsets of the Euclidean space of positive upper Banach density, where one is concerned with finding all large dilates of a fixed geometric configuration.  The upper Banach density of a measurable set $A\subseteq \R^d$ is defined as
\[\overline{\delta}(A) :=  \limsup_{N\to \infty} \sup_{x\in \R^d} \frac{|A\cap (x + [0,N]^d)|}{N^d} >0.\]
The simplest example is a two-point pattern, for which it holds that whenever $A\subseteq \R^2$ satisfies $\overline{\delta}(A)>0$, then there is a threshold $\lambda_0=\lambda_0(A)$ such that for each $\lambda \geq \lambda_0$, $A$ contains points $x,y$ with $|x-y|=\lambda$.  
This was conjectured by Sz{\'e}kely \cite{Sz83}, the first proof   was announced by  Furstenberg, Katznelson and Weiss \cite{FKW},  while in the meantime it was also independently shown in  \cite{Bou86} and by Falconer and Marstrand \cite{FM}.
The paper \cite{Bou86} also  generalizes this result for a two-point pattern and  shows existence of isometric copies of all large dilates  a fixed non-degenerate $(k+1)$-point simplex,  in sets of positive upper Banach density in $\R^{k+1}$. 
The paper  \cite{LM20}   extends this result to distance graphs, while the papers \cite{LM18, LM19} extend it also to products of non-degenerate simplices. More precisely, in \cite{LM19}  it is shown that if one is given $\Delta_1,\ldots, \Delta_n$ as in Theorem \ref{mainthm} and  $A\subseteq \R^{k_1+1}\times \cdots \times \R^{k_n+1}$  satisfies $\overline{\delta}(A)>0$,
then there exists $\lambda_0 = \lambda_0(A,\Delta_1,\ldots, \Delta_n)>0$ such that for all $\lambda \geq \lambda_0$,  the set $A$ contains $\Delta_1'\times \cdots \times \Delta_n'$,  where each $\Delta_i'$  an isometric copy of $\lambda \Delta_i$. 

The corresponding result for all large dilates of rectangular boxes in sets of positive upper Banach density was also obtained in \cite{DK18} using harmonic analysis techniques,  but with some loss in the dimension of the ambient space. As commented in \cite{DK20}, the arguments therein also give the corresponding result for cubes in sets of positive upper Banach density  and they extend to boxes as well. Similarly, the techniques used in this paper   can be extended to reprove the full  result from \cite{LM19} in sets of positive upper Banach density, provided one shows $L^p$ boundedness for the forms associated with hypergraphs, which arise in the error part, or cancellation estimates to arbitrary, not necessarily consecutive scales.

Recently, Kova\v{c} \cite{Kov20}   initiated the study of density theorems for anisotropic point configurations. The paper \cite{Kov20} established existence of all large dilates of
anisotropic simplices, anisotropic boxes, and an anisotropic version of the result on distance trees, respectively,  extending several results contained in  \cite{Bou86, LM19}, and  \cite{LM20} to the anisotropic setting. Proving   anisotropic versions of Theorem \ref{mainthm} and the corresponding results in sets of positive density is yet another possible line of investigation.\\

 {\em Notation:} For two non-negative quantities $A,B$ we write $A\lesssim B$ if there exists a constant $C>0$ such that $A\leq C B$. We write $A\sim B$ if both $A\lesssim B$ and $B\lesssim A$.  In this paper, all constants are allowed to depend on  the dimensions    $k_i,n$, and simplices $\Delta_i$ and we will not explicitly denote that in the notation.

\section{Counting forms and their decomposition}
\label{sec:forms}
We begin this section by defining counting forms which detect isometric copies of dilates of products of fixed non-degenerate simplices in a subset of the unit cube. 

Given   $\Delta = \{0, v^1,\ldots, v^{d}\}$, where $v^1,\ldots, v^d$ are linearly independent vectors in $\R^d$, we shall view them as vertices of a simplex in  $\R^{d+1}$ by identifying $\R^d$ with $\R^d\times \{0\} \subseteq \R^{d+1}$. 
First we recall the definitions of suitable spherical measures associated with  $\lambda \Delta$, $\lambda>0$, as defined in \cite{Kov20} and previously used   in \cite{Bou86, LM18, LM19}.
 
For each $1\leq j \leq d$, let  us write $v^j = |v^j| u^j$   where   $u^j$ are unit vectors. For $2\leq j \leq d$,  the orthogonal projection of $u^j$ onto $\textup{span}\{u^1,\ldots, u^{j-1}\}$ is given by 
\[\beta_{j,1}u^1+\beta_{j,2}u^2+\ldots \beta_{j,j-1}u^{j-1}\]
for some $\beta_{j,1},\ldots, \beta_{j,j-1}\in \R$. Let $\widetilde{\sigma}$ denote the normalized surface measure on the unit sphere $S^{d}$ in  $\R^{d+1}$. For $2\leq j \leq d$ let 
$$\widetilde{\sigma}^{\Delta, x^1,\ldots, x^{j-1}}$$
be  the normalized surface measure on the $(d-j+1)$-dimensional sphere  in $\R^{d+1}$ centered at
\[\beta_{j,1}\frac{x^1}{|x^1|}+\beta_{j,2}\frac{x^2}{|x^2|}+\ldots \beta_{j,j-1}\frac{x^{j-1}}{|x^{j-1}|}\]
with radius being equal to 
$\textup{dist}(u^j,\textup{span}\{u^1,\ldots, u^{j-1}\}),$ which lies in the $(d-j+2)$-dimensional plane in $\R^{d+1}$ that is orthogonal to $\textup{span}\{u^1,\ldots, u^{j-1}\}$. 

For a Borel measure $\mu$ and a  Borel set $A$  
we define
$$\mu_\lambda(A) := \mu(\lambda^{-1}A)$$ 
and we write $\sigma^{\Delta}:= \widetilde{\sigma}_{|v^1|}$, $\sigma^{\Delta,x^1,\ldots, x^{j-1}}:= \widetilde{\sigma}^{\Delta,x^1,\ldots, x^{j-1}}_{|v_j|} $. 
We also define the measure $\sigma_{\Delta,\lambda}$ by 
\[\sigma_{\Delta,\lambda}(A) :=  \int_{\R^{d(d+1)}}  \mathbbm{1}_A(y^1,\ldots, y^{d}) \, d\sigma_{\lambda}^{\Delta,y^1,\ldots,y^{d-1}}(y^{d})\cdots d\sigma_{\lambda}^{\Delta,y^1}(y^2)\,d\sigma_{\lambda}^{\Delta}(y^1) \]
where   $A$ is a Borel subset of $\R^{d(d+1)}$. The support of  $\sigma_{\Delta,\lambda}$ is the set of all points $y$ such that $\{0,y^1,\ldots, y^{d}\}$ is isometric to $\lambda \Delta$. Note that   the definition of $\sigma_{\Delta,\lambda}$  is symmetric in the choice of the order of integration in  $y^1,\ldots, y^{d}$. This 
can be seen by rewriting    $\sigma_{\Delta,\lambda}$ using the Haar measure on  $SO(d+1,\R)$, see for instance formulae (1) and (9) in  \cite{Bou86}. By $d\sigma_{\Delta,\lambda}(\cdot -c)$ we will denote integration with respect to the measure $\sigma_{\Delta,\lambda}$ translated by $c\in \R^{d(d+1)}$.

Let us fix  $n\geq 1$ and the dimensions $k_i\geq 1$ for $1\leq i \leq n$. We also fix  vertices of non-degenerate simplices of $k_i$ points $\Delta_i$ in $\R^{k_i+1},\,1\leq i \leq n$. 

 For any tuple of integers $L=(l_1,\ldots, l_n)$ satisfying $1\leq l_i\leq k_i$ for each $1\leq i \leq n$,  let 
$\mathcal{H}_L$ be the set of all functions $h\colon~\{ 1, \ldots, n \} \rightarrow \N\cup \{0\}$ satisfying $h(i) \leq l_i$ for each $1\leq i \leq n$. We will write $L+1:= L+(1,\ldots, 1)$.  
We will also denote 
\[K:=(k_1,\ldots , k_n), \quad k:=k_1+\cdots + k_n, \quad  \textup{and} \quad \H:=\H_K. \]
For $h \in\H_L$ and 
\[x=  (x^0_1,\ldots, x^{l_1}_1, x_2^0,\ldots x_2^{l_2}, \ldots, x_n^0,\ldots, x_n^{l_n})\in  \R^{(K+1)\cdot (L+1)}\] we define $\Pi_h:\R^{(K+1)\cdot (L+1)}\to \R^{k+n}$ by 
$$\Pi_h {x} := (x^{h(1)}_1,\ldots,x^{h(n)}_n).$$ 
 
Now we are ready to define the counting forms which detect isometric copies of dilates of our simplices. 
For $\lambda>0$ and a measurable set $A\subseteq \R^{k+n}$  
we define the form  
\begin{align*}
\mathcal{N}^0_{\lambda}(A) =
\int_{\R^{(K+1) \cdot (K+1)}}  \Big( \prod_{h \in \H} \mathbbm{1}_{A}(\Pi_h x) \Big ) 
  \prod_{i=1}^n d\sigma_{\Delta_i,\lambda} ((x_i^{r}-x_i^0)_{r=1}^{k_i}) \, dx_i^0.
\end{align*}
Note that if 
 $\mathcal{N}^0_{\lambda}(A) > 0,$
then $A$ contains   $\Delta_1'\times \cdots \times \Delta_n'$,  where each $\Delta_i'$ is an isometric copy of the dilated simplex $\lambda \Delta_i$.

If $\varphi$ is a function on $\R^m$ and $t\in \R$, we denote  $\varphi_t(x):=t^{-m}\varphi(t^{-1}x)$. 
Let $g$ be    the   Gaussian
\[g(x):=e^{-\pi |x|^2}.\] The  dimension of Gaussians will always be understood from the context and will not be part of the notation.   For $\varepsilon>0$ we define a smoothed form  
\[
\mathcal{N}^{\varepsilon}_{\lambda}(A) :=
\int_{\R^{(K+1) \cdot (K+1)}  }  \Big( \prod_{h \in \H} \mathbbm{1}_{A}(\Pi_h x) \Big ) \Big(  \prod_{i=1}^n {\sigma}_{\Delta_i,\lambda} \ast {g}_{\varepsilon\lambda} ((x_i^{r}-x_i^0)_{r=1}^{k_i}) \Big ) \, dx . \]
In the limit as $\varepsilon \to 0$ it recovers   $\mathcal{N}^{0}_{\lambda}(A)$. This can be seen similarly  as in \cite[Section 3]{DK20}. 

Theorem \ref{mainthm} will be deduced from the following three lemmas. 
\begin{lemma}\label{strlemma}  
There exists a constant $C>0$ such that for any $\lambda\in (0,1]$, $\delta \in (0,1/2]$, and  a measurable set  $A\subseteq [0,1]^{k_1+1}\times \cdots \times [0,1]^{k_n+1}$    with $|A|\geq \delta$, 
\begin{equation}\label{lemma1}
\mathcal{N}^{1}_{\lambda}(A) \geq C  \delta^{(k_1+1)(k_2+1)\cdots (k_n+1)}.
\end{equation}
\end{lemma}
\begin{lemma}\label{errlemma}
There exists a constant $C>0$ such that for any    $\varepsilon\in (0,1]$, positive integer    $J$, any $ \lambda_j\in (2^{-j},2^{-j+1}]$,  $1\leq j \leq J$,   and any measurable set  $A\subseteq [0,1]^{k_1+1}\times \cdots \times [0,1]^{k_n+1}$, 
\begin{equation}\label{lemma2}
\sum_{j=1}^{J} \big|\mathcal{N}^{\varepsilon}_{\lambda_j}(A)-\mathcal{N}^{1}_{\lambda_j}(A)\big| \leq C J^{1-2^{-n}} \varepsilon^{-1-\sum_{i=1}^n k_i(k_i+2)}.
\end{equation}
\end{lemma}
 
\begin{lemma}\label{unilemma}
There exists a constant $C>0$ such that for any $\lambda,\,\varepsilon \in (0,1]$ and any measurable set $A\subseteq [0,1]^{k_1+1}\times \cdots \times [0,1]^{k_n+1}$, 
\begin{equation}\label{lemma3}
\big|\mathcal{N}^{0}_{\lambda}(A)-\mathcal{N}^{\varepsilon}_{\lambda}(A)\big| \leq C \varepsilon^{\frac{1}{2}}.
\end{equation}
\end{lemma}

Let us first show how to deduce Theorem \ref{mainthm} from these three lemmas. We  decompose
$$\mathcal{N}^0_{\lambda}(A) = \mathcal{N}^1_{\lambda}(A) + (\mathcal{N}^{\varepsilon}_{\lambda}(A) - \mathcal{N}^1_{\lambda}(A)) + (\mathcal{N}^0_{\lambda}(A) - \mathcal{N}^{\varepsilon}_{\lambda}(A)).$$
Let 
$C_1$-$C_3$ be the constants that appear in Lemmas \ref{strlemma}-\ref{unilemma}, respectively. We may assume $C_1\in (0,1], C_2,C_3\in [1,\infty)$.  Let $\delta \in (0,1/2]$,  $\kappa:=(k_1+1)\cdots (k_n+1)$, and $\varrho := \sum_{i=1}^n k_i(k_i+2)$. 
Take  
$$\varepsilon := (C_1 \delta^{\kappa})^2(3C_3)^{-2}, \quad J := \lfloor ( 3C_2C_1^{-1}\varepsilon^{-\varrho-1}\delta^{-\kappa} )^{2^n} \rfloor+1. $$
Note that 
$J \lesssim  \delta^{-\kappa 2^n(2\varrho+3)},$
which gives
$$2^{-J} \geq 2^{-C\delta^{-\kappa 2^n(2\varrho+3)}} = e^{-C'\delta^{-\kappa 2^n(2\varrho+3)}}$$
for some constants $C$ and $C'$.

By \eqref{lemma2}, there exists $1\leq j \leq J$ such that for each $\lambda \in ( 2^{-j}, 2^{-j+1}]$, 
$$\big|\mathcal{N}^{\varepsilon}_{\lambda}(A)-\mathcal{N}^{1}_{\lambda}(A)\big| \leq C_2 J^{-2^{-n}} \varepsilon^{-\varrho-1} \leq \frac{C_1}{3} \delta^{\kappa}.$$
Combined with \eqref{lemma1} and \eqref{lemma3}, we obtain the desired estimate
\[\mathcal{N}^0_{\lambda}(A) \geq \mathcal{N}^{1}_{\lambda}(A) - \big|\mathcal{N}^{\varepsilon}_{\lambda}(A)-\mathcal{N}^{1}_{\lambda}(A)\big|- \big|\mathcal{N}^{0}_{\lambda}(A)-\mathcal{N}^{\varepsilon}_{\lambda}(A)\big|\]
\[\geq C_1 \delta^{\kappa} - \frac{C_1}{3} \delta^{\kappa} - \frac{C_1}{3} \delta^{\kappa} =  \frac{C_1}{3} \delta^{\kappa} > 0.\]

\section{The structured part: Proof of lemma \ref{strlemma}}
Let $\delta \in (0,1/2]$ and  $A\subseteq [0,1]^{k_1+1}\times \cdots \times [0,1]^{k_n+1}$ with $|A|\geq \delta$. 
Let $m$ be the unique positive integer such that $\lambda \in  (2^{-m}, 2^{-m+1}]$. For each  $1\leq i\leq n$,     let $\mathcal{Q}_m^{k_i}$ be a collection of all dyadic cubes   in $[0,1]^{k_i+1}$ with side length $2^{-m}$. We write $\mathcal{Q}_m := \mathcal{Q}_m^{k_1}\times \cdots \times \mathcal{Q}_m^{k_n}$.

We  partition the domain of integration of each $x_i^{r_i}$ into congruent cubes of side length $2^{-m}$ and restrict the domain of integration so that for each $i$, all $x_i^{r_i}$, $1\leq r_i \leq k_i$,   lie in the same  cube from  $\mathcal{Q}_m$.   This yields
\[
\mathcal{N}^1_{\lambda}(A) \geq 2^{-D}  \sum_{(Q_1, \ldots , Q_n) \in \mathcal{Q}_m} \fint_{Q_1^{k_1+1}\times \cdots \times  Q_n^{k_n+1}}    \Big( \prod_{h\in \H} \mathbbm{1}_A(\Pi_h x) \Big )    \prod_{i=1}^n  \sigma_{\Delta_i,\lambda}*g_\lambda((x_i^r-x_i^0)_{r=1}^{k_i}) \,  dx\]
 where $D = m (K+1) \cdot (K+1)$ and $Q^{k_i}_i$ denotes the $k_i$-fold Cartesian product of $Q_i$.

Since the measure $\sigma_{\Delta_i,\lambda }$ is  supported in the ball $B(0, \lambda C_{\Delta_i})$ in $\R^{k_i(k_i+1)}$  of radius $\lambda C_{\Delta_i}$ centered at  the origin,  where $C_{\Delta_i}$  is   the diameter of  $\Delta_i$, 
for each $1\leq i \leq n$ and $y\in B(0,\lambda R)$,  
\[\sigma_{\Delta_i,\lambda} * g_\lambda (y) = \lambda^{-k_i(k_i+1)} \int_{\R^{k_i(k_i+1)}}e^{-\pi |\lambda^{-1} (y-u)|^2}  d\sigma_{\Delta_i,\lambda}(u)   \]
\[\ge \lambda^{-k_i(k_i+1)}  e^{-\pi (C_{\Delta_i}+R)^2}  \int_{\R^{k_i(k_i+1)}} d\sigma_{\Delta_i,\lambda}(u) =  \lambda^{-k_i(k_i+1)}  e^{-\pi (C_{\Delta_i}+R)^2} 
\]
The last equality follows after integrating in $u^{k}, u^{k-1}$, all the way to $u^1$, where $u=(u^1,\ldots, u^k)$, using that the spherical measures are normalized.
Therefore, with $R^2=k_i(k_i+1)$  we obtain  
$$\sigma_{\Delta_i,\lambda}  \ast g_{\lambda} \gtrsim  \lambda^{-k_i(k_i+1)}\mathbbm{1}_{[-\lambda,\lambda]^{k_i(k_i+1)}} \gtrsim  2^{mk_i(k_i+1)} \mathbbm{1}_{[-2^{-m},2^{-m}]^{k_i(k_i+1)}}.$$

Estimating $\sigma_{\Delta_i,\lambda} * g_\lambda$ from below for each $i$ we obtain
\[
\mathcal{N}^1_{\lambda}(A) \gtrsim  2^{-m(k+n)} \sum_{(Q_1, \ldots , Q_n) \in \mathcal{Q}_m} \fint_{Q_1^{k_1+1}\times \cdots \times  Q_n^{k_n+1}} \Big(\prod_{h\in \mathcal{H}}\mathbbm{1}_{A}(\Pi_h x)  \Big ) 
\]
\[ \times  \prod_{i=1}^n \mathbbm{1}_{[-2^{-m},2^{-m}]^{k_i(k_i+1)}}((x_i^r-x_i^0)_{r=1}^{k_i})\, dx.    \]
Since for $1\leq r_i\leq k_i$ both $x_i^{r_i}$ and $x_i^0$ belong to the same cube $Q_i$, this gives
\[
\mathcal{N}^1_{\lambda}(A) \gtrsim  2^{-m(k+n)} \sum_{(Q_1, \ldots , Q_n) \in \mathcal{Q}_m} \fint_{Q_1^{k_1+1}\times \cdots \times  Q_n^{k_n+1}} \prod_{h\in \mathcal{H}}\mathbbm{1}_{A}(\Pi_h x) \, dx.
\]

To complete the proof of the estimate \eqref{lemma1}, we claim that 
\begin{equation}
\label{lowerbd}
 \fint_{Q_1^{k_1+1}\times \cdots \times  Q_n^{k_n+1}}  \prod_{h\in \mathcal{H}}\mathbbm{1}_{A}(\Pi_h x) \,  dx \geq \bigg( \fint_{Q_1\times \cdots \times Q_n} \mathbbm{1}_{A} \bigg)^{\kappa} 
\end{equation}
where $\kappa =(k_1+1)\cdots (k_n+1)$.  
This claim then implies 
\[\mathcal{N}^1_{\lambda}(A) \gtrsim  2^{-m(k+n)} \sum_{(Q_1, \ldots, Q_n) \in \mathcal{Q}_m} \bigg( \fint_{Q_1\times\cdots\times Q_n} \mathbbm{1}_{A} \bigg)^{\kappa} \]
\[\geq \bigg( 2^{-m(k+n)} \sum_{(Q_1, \ldots , Q_n) \in \mathcal{Q}_m} \fint_{Q_1\times\cdots\times Q_n}  \mathbbm{1}_{A} \bigg)^{\kappa} = \bigg( \int_{[0,1]^{k+n}} \mathbbm{1}_{A} \bigg)^{\kappa} \geq \delta^{\kappa}.
\]
The second inequality follows by applying Jensen's inequality for the normalized counting measure over the set of cardinality $2^{m(k+n)}$.
This finishes the proof the lemma, up to verification of the bound  \eqref{lowerbd}.

The estimate \eqref{lowerbd} can be proven by induction on $n$. We rewrite
\[ \fint_{Q_1^{k_1+1}\times \cdots \times Q_n^{k_n+1}}  \prod_{h\in \mathcal{H}}\mathbbm{1}_{A}(\Pi_h x)  \, dx = \fint_{Q_1^{k_1+1}\times \cdots \times Q_n^{k_n+1}}  \prod_{a=0}^{k_n} \prod_{h\in \H: h(n)=a} \mathbbm{1}_A (\Pi_h x)  \, dx \]
\begin{equation}
    \label{unifbd}
    = \fint_{Q_1^{k_1+1}\times \cdots \times Q_{n-1}^{k_{n-1}+1}} \Big( \fint_{Q_n} \prod_{h\in \H: h(n)=0} \mathbbm{1}_A (\Pi_h x)dx_n^0 \Big )^{k_n+1} dx' 
\end{equation}
where $x' = (x^0_1,\ldots, x^{k_1}_1, x_2^0,\ldots x_2^{k_2}, \ldots, x_{n-1}^0,\ldots, x_{n-1}^{k_{n-1}})$.
From here we see that the claim  holds for $n=1$ since in this case,   \eqref{lowerbd} is an equality. 
For the induction step we 
use Fubini and Jensen's inequality for the integral over $x'$, which bounds \eqref{unifbd} from below by
\[ \Big( \fint_{Q_n} \fint_{Q_1^{k_1+1}\times \cdots \times Q_{n-1}^{k_{n-1}+1}} \prod_{h\in \H: h(n)=0} \mathbbm{1}_A (\Pi_h x)\, dx'\, dx_n^0 \Big )^{k_n+1}. \]
By the induction hypothesis applied to the set
\[\{ y \in \mathbb{R}^{k+n-k_n-1} : (y,x_n^0) \in A \}\]
for a fixed $x_n^0$, we estimate the last expression from below by
\[ \bigg( \fint_{Q_n} \bigg( \fint_{Q_1\times \cdots \times Q_{n-1}} \mathbbm{1}_{A}(y)dy \bigg)^{(k_1+1) \cdots (k_{n-1}+1)}dx_n^0 \bigg)^{(k_n+1)}.\]
Another application of Jensen's inequality, this time for the integral over $x_n^0$, gives \eqref{lowerbd}. This completes the proof of Lemma \ref{errlemma}.

\section{The error part: Proof of lemma \ref{errlemma}}
The key  ingredient in the proof of bounds for the error part are cancellation estimates for forms similar to multilinear singular integrals associates with $n$-regular $n$-partite hypergraphs. These are the content of     Lemma \ref{lemma:ind}. 

For $1\leq v \leq n$, parameters  $\lambda,a,b>0$,  and  $A\subseteq [0,1]^{k_1+1}\times \cdots \times [0,1]^{k_n+1}$ we introduce 
\[N(v, \lambda, a,b, A) := \int_{a}^{b} \int_{\R^{(K+1) \cdot (K+1)}  }  \Big( \prod_{h \in \H} \mathbbm{1}_{A}(\Pi_h x) \Big )  \]
\begin{equation}\label{Nform}
 \times   {\sigma}_{\Delta_{v},\lambda} \ast {(\Delta g)}_{t\lambda} ((x_v^r-x_v^0)_{r=1}^{k_v})   
 \prod_{i=1, i\neq v}^n {\sigma}_{\Delta_i,\lambda} \ast {g}_{t\lambda} ( (x_i^r-x_i^0)_{r=1}^{k_i}) \, dx \, \frac{dt}{t}.
\end{equation}
 Here, $\Delta g$ is the Laplacian of $g$ in $\R^{{k_v}(k_{v}+1)}$ dimensions.  
If we  further write  
\begin{equation}
    \label{laplaciansplit}
\Delta g ((x_v^r-x_v^0)_{r=1}^{k_v})=\sum_{w=1}^{k_{v}} \Delta g (x_{v}^{w}-x_{v}^0)  \,  g((x_v^r-x_v^0)_{r=1,r\neq w}^{k_v}),
\end{equation}
  and expand out the convolutions, we obtain 
 \[N(v, \lambda, a,b, A) =  \sum_{w=1}^{k_{v}}  \int_{a}^{b} \int_{\R^{(K+1) \cdot (K+1) + k_{v}(k_v+1)} }  \Big( \prod_{h \in \H} \mathbbm{1}_{A}(\Pi_h x) \Big )   \]
\[ \times (\Delta g)_{t\lambda} (x_{v}^{w} - x_{v}^0 - p_{v}^{w}) \,  g_{t\lambda} ((x_{v}^r - x_{v}^0 - p_{v}^r)_{r=1,r\neq w}^{k_v})  \]
\begin{equation}
    \label{Nform-2}
    \times \Big( \prod_{r=1}^{k_v} d\sigma_{\lambda}^{\Delta_{v},p_{v}^{1},\ldots p_{v}^{k_v-r}}(p_v^{k_v+1-r}) \Big ) \prod_{i=1, i\neq v}^n {\sigma}_{\Delta_i,\lambda} \ast {g}_{t\lambda} ((x_i^r-x_i^0)_{r=1}^{k_i}) \, dx \, \frac{dt}{t}. 
\end{equation}

Using the fundamental theorem of calculus in $t$, the Leibniz rule and the heat equation
 \begin{equation}
 \label{heateqn}
 (\Delta g)_{t\lambda} = 2\pi t \partial_t(g_{t\lambda})
 \end{equation}
we thus obtain the identity 
\[ \sum_{j=1}^J (\mathcal{N}^{1}_{\lambda_j}(A)-\mathcal{N}^{\varepsilon}_{\lambda_j}(A))  = \frac{1}{2\pi}\sum_{v=1}^n \sum_{j=1}^J N(v, \lambda_j, \varepsilon,1,A).\]
Due to symmetry with respect to the simplices and  the definitions of the measures $\sigma_{\Delta_i,\lambda}$,  it suffices to prove the corresponding bound from Lemma \ref{lemma2} for $v=1$ and $w = k_1$ in \eqref{Nform-2}. 

Now, for $j \in \mathbb{N}$, $t \in ( 0,\infty )$ and $s \in [ 2^{-j-5}t, 2^{-j-4}t ]$ let
$$r_j(s,t) := \sqrt{t^2\lambda_j^2-s^2},\quad c_j(s,t) := \frac{t^2\lambda_j^2}{sr_j(s,t)}.$$
We can see that $s \sim 2^{-j}t \sim t\lambda_j \sim r_j(s,t)$ and  thus $c_j(s,t) \sim 1$. Also, for a $d$-dimensional Gaussian $g$  we have 
\begin{equation}
    \label{gkh}
    (\Delta g)_{t\lambda_j} = c_j(s,t) \sum_{m =1}^{d} (\partial_m g)_{r_j(s,t)} \ast (\partial_m g)_s.
\end{equation}
We will also split and estimate
\begin{equation}
\label{fest}    
\prod_{h\in \H} \mathbbm{1}_A (\Pi_h x) = \prod_{a=0}^{k_1} \prod_{h\in \H: h(1)=a}  \mathbbm{1}_A (\Pi_h x) \leq \prod_{h\in \H: h(1)=k_1}  \mathbbm{1}_A (\Pi_h x)
\end{equation}

We start by multiplying the integrand \eqref{Nform-2} by
\[(\log 2)^{-1}\int_{2^{-j-5}}^{2^{-j-4}}\frac{ds}{s} =1.\]
Applying the triangle inequality and using   \eqref{gkh} and \eqref{fest}, we estimate $\sum_{j=1}^J|N(1,\lambda_j, \varepsilon,1,A) |$  up to an absolute constant by  
\[\sum_{j=1}^J \sum_{m=1}^{k_1+1}\int_\varepsilon^1 \int_{2^{-j-5}t}^{2^{-j-4}t} \int_{\R^{D}}  \mathbbm{1}_{[0,1]^{k+n}}(x_1^0,\ldots, x_n^0) \,   \Big |  \int_{\R^{k_1+1}} \Big( \prod_{h\in \H: h(1)=k_1}  \mathbbm{1}_A (\Pi_h x)\Big)     \]
\[ (\partial_m g)_s (x_1^{k_1}-q)\,  dx_1^{k_1} \Big |\times (\sigma_{\lambda_j}^{\Delta_1,p_1^1,\ldots,p_1^{k_1-1}}  \ast |\partial_m g|_{r_j(s,t)})(x_1^0-q)\, {g}_{t\lambda_j}((x_1^{r}-x_1^0-p_1^{r})_{r=1}^{k_1-1})   \]
\begin{equation*}
    \times 
    \Big( \prod_{r=1}^{k_v} d\sigma_{\lambda}^{\Delta_{v},p_{v}^{1},\ldots p_{v}^{k_v-r}}(p_v^{k_v+1-r}) \Big )\, d((x_1^r)_{r=0}^{k_1-1}) \,
    dq 
\end{equation*}
\begin{equation}\times  \Big ( \prod_{i=2}^n   \sigma_{\Delta_i,\lambda_j} * {g}_{t\lambda_j}((x_i^r-x_i^0)_{r=1}^{k_i}) d((x_i^r)_{r=0}^{k_i}) 
\Big )\,  \frac{ds}{s} \, \frac{dt}{t}   
\label{gausstail}
\end{equation}
where now $D=(K+1) \cdot (K+1) + (k_1-1)(k_1+1)$.

If $\mu$ is a probability measure on $\R^d$, $d\geq 2$,  whose support is contained in a sphere in $\R^d$ and whose radius depends only on the simplices $\Delta_i$,  and  $\phi$  a Schwartz function on $\R^{d}$, for $t\in [\varepsilon,1]$ we can estimate  
\begin{equation}
    \label{estmeasure}
    (\mu_{t^{-1}} \ast \phi)(y) \lesssim  \int_{\mathbb{R}^{d}} ( 1+ | y-ut^{-1} | )^{-d-1}d\mu(u) \lesssim \varepsilon^{-d-1}(1+|y|)^{-d-1}
\end{equation}
with the implicit constant depending on $\phi$.
The second inequality follows from 
\[1+|y| \leq  1+|y-ut^{-1}| + |u|t^{-1}  \lesssim \varepsilon^{-1}( 1+|y-ut^{-1}|).\]
We can rescale   inequality \eqref{estmeasure} by $t\lambda_j s^{-1} \sim 1$ in order to get
$$(\mu_{\lambda_js^{-1}} \ast \phi_{t\lambda_js^{-1}})(y) \lesssim  \varepsilon^{-d-1}(1+|y|)^{-d-1}.$$
Analogously we obtain
$$(\mu_{\lambda_js^{-1}} \ast \phi_{r_j(s,t)s^{-1}})(y) \lesssim \varepsilon^{-d-1}(1+|y|)^{-d-1}.$$
Dominating the right-hand side by a superposition of Gaussians in $\R^d$
\[\int_1^\infty g_{\beta}(y) \, \frac{d\beta}{\beta^2} \sim (1+|y|)^{-d-1}   \]
and rescaling 
 by $s$ gives
 \begin{equation}
     \label{guassianineq}
     (\mu_{\lambda_j} \ast \phi_{t\lambda_j})(y) \lesssim  \varepsilon^{-d-1}\int_1^{\infty} {g}_{\beta s}(y)\, \frac{d\beta}{\beta^2}, \quad (\mu_{\lambda_j} \ast \phi_{r_j(s,t)})(y) \lesssim \varepsilon^{-d-1}\int_1^{\infty} {g}_{\beta s}(y)\, \frac{d\beta}{\beta^2}.
 \end{equation}

Now we apply the second inequality in \eqref{guassianineq} to the convolution 
\[\sigma_{\lambda_j}^{\Delta_1,p_1^1,\ldots,p_1^{k_1-1}} \ast |\partial_m g|_{r_j(s,t)}(x_1^0-q)\] in the expression \eqref{gausstail}.  After that we can recognize another convolution, this time after integrating in the variable $p_1^{k_1-1}$, i.e. 
\[\sigma_{\lambda_j}^{\Delta_1,p_{1}^{1},\ldots p_1^{k_1-2}}*g_{t\lambda_j}(x_1^{k_1-1}-x_1^0). \]
We apply the first inequality in \eqref{guassianineq} to this convolution. 
We continue the same procedure in the variables $p_{1}^{k_1-2},\ldots, p_1^1$  until we have estimated all convolutions in these variables one by one. Then we also use  \eqref{guassianineq} on the remaining  convolutions $\sigma_{\Delta_i,\lambda_j}*g_{t\lambda_j}$.  If we denote $\beta=(\beta_1,\ldots, \beta_n)$ and  $\beta_i=(\beta_{i}^0,  \beta_{i}^1,\ldots,\beta_{i}^{k_i})\in (1,\infty)^{k_i+1}$ for $1\leq i \leq n$,
we obtain an estimate for  \eqref{gausstail} by 
\[ \varepsilon^{-\varrho}\, \sum_{j=1}^J \sum_{m=1}^{k_1+1} \int_{(1,\infty)^{k+n }} \int_{\varepsilon}^1 \int_{2^{-j-5}t}^{2^{-j-4}t} \int_{\mathbb{R}^{(K+1)\cdot (K+1)}}  \mathbbm{1}_{[0,1]^{k+n}}(x_1^0,\ldots, x_n^0)  \]
\begin{equation}
\label{form:tobound}
\times \Big |  \int_{\R^{k_1+1}}    \Big( \prod_{h\in \H: h(1)=k_1}  \mathbbm{1}_A (\Pi_h x)\Big) (\partial_m g)_{s} (x_1^{k_1}-q) dx_1^{k_1} \Big |  \,  d\mu^{(1,k_1)}_{K,g,\beta}\, \frac{ds}{s} \, \frac{dt}{t} \,   \mathcal{D}(\beta)^{-2}  \,d\beta 
\end{equation}
where $\varrho$ is  as in Lemma \ref{errlemma} and $\mathcal{D}(\beta)$ is just the product of all $\beta_i^r$  over $1\leq i\leq n$ and $0\leq r\leq k_i$. 
Here, for $L=(l_1,\ldots,l_n)$, 
 a tuple $\alpha=(\alpha_1,\ldots,\alpha_n)$, $\alpha_i=(\alpha_{i}^0,\alpha_{i}^1,\ldots,\alpha_{i}^{l_i})\in (1,\infty)^{l_i+1}$,   a Schwartz function $\varphi$, $1\leq v\leq n$,  and $1\leq w \leq l_v$ we set
\[d{\mu}^{(v,w)}_{L,\varphi,\alpha} = \varphi_{s\alpha_{v}^0} (x_{v}^0-q) \Big( \prod_{r=1,r\neq w}^{l_{v}} g_{s\alpha_{v}^r}(x_{v}^{r}-x_{v}^0)  dx_{v}^{r} \Big )   \Big( \prod_{i=1, i\neq v}^n g_{s\alpha_{i}} ((x_i^r-x_i^0)_{r=1}^{l_i}) dx_i 
\Big ) \, dq,\]
and we  set $x_i=(x_i^0,\ldots,x_i^{l_i})$. 
Note that the integrand in \eqref{form:tobound} is constant in $\beta_{1}^{k_1}$ and $\beta_{i}^0$, $i\neq v$,  
but we added 
 an integral in these variables that  equals to one for uniform notation. 
 
 Now,  for $1\leq v\leq n$, $1\leq m \leq k_{v}+1$,   any tuple $\alpha=(\alpha_1,\ldots, \alpha_n)$ as above,  any real numbers $a,b$ with $b/a\geq 2$, any $L=(l_1,\ldots, l_n)$ with $1\leq l_i \leq k_i$, and  $1\leq w \leq l_v$,    we   define
 \[\widetilde{\Theta}^{(v,w,m)}_{L,\varphi,a,b, \alpha}(A) :=   \int_{a}^{b} \int_{\mathbb{R}^{(K+1)\cdot (L+1)}}    \mathbbm{1}_{[0,1]^{k+n}}(x_1^0,\ldots, x_n^0)     \]
 \begin{equation}
 \label{thetatilde}\times  \Big| \int_{\R^{k_{v}+1}} \Big( \prod_{h\in \H_L: h(v) = w} \mathbbm{1}_A(\Pi_h x) \Big ) (\partial_{m} g)_{s\alpha_{v}^w}(x_{v}^{w} -q ) dx_{v}^{w} \Big | \, d\mu^{(v,w)}_{L,\varphi,\alpha} \, \frac{ds}{s}.
 \end{equation}
Then, after summing in $j$, \eqref{form:tobound} can be recognized as 
 \begin{equation}
 \label{mainbound}
 \varepsilon^{-\varrho} \,\sum_{m=1}^{k_1+1} \int_{(1,\infty)^{k+n}} \int_{\varepsilon}^1 \widetilde{\Theta}_{K,g,2^{-J-5}t, 2^{-5}t,\widetilde{\beta}}^{(1,k_1,m)}(A)\, \frac{dt}{t}\, \mathcal{D}(\beta)^{-2} \,d\beta 
 \end{equation}
 where $\widetilde{\beta}_{1}^{k_1}=1$, $\widetilde{\beta}_{i}^0=1$ for all $i\neq v$, 
 and $\widetilde{\beta}$ agrees with $\beta$ in all other coordinates. 
Thus it suffices to bound the form inside the inner integral for each fixed $t$ and $\beta_{i}^{k_i}$.

In order to do that, for the same range of parameters $v,L,a,b,\alpha$ as in the definition of \eqref{thetatilde} we define yet another form
 \[\Theta^{(v)}_{L,a,b,\alpha}(A) :=  - \int_{a}^{b} \int_{\mathbb{R}^{(K+1)\cdot (L+1)}}      \Big( \prod_{h\in \H_L} \mathbbm{1}_A(\Pi_h x) \Big )  \]
\begin{equation}
\times (\Delta g)_{s\alpha_{v}} ((x_{v}^{r}-x_{v}^{0})_{r=1}^{l_v}  ) \prod_{i=1, i\neq v}^n g_{s\alpha_{i}} ((x_{i}^{r}-x_{i}^{0})_{r=1}^{l_i} )\, dx\, \frac{ds}{s}.  
\label{theta}
\end{equation}
Note that as opposed to \eqref{form:tobound}, this form is not necessarily non-negative.  
The bound for \eqref{form:tobound} will  follow from the following crucial lemma, which establishes  a cancellation estimate for both \eqref{thetatilde} and \eqref{theta}.  In the  proof we inductively simplify the kernels of the singular forms using repeated applications of Cauchy-Schwarz,  partial integration, and positivity arguments.

 \begin{lemma} \label{lemma:ind} 
For any $L,\varphi,a,b,\alpha,v,w,m$ as above, 
\begin{equation} \label{thetabd}
    |\Theta^{(v)}_{L,a,b,\alpha}(A)|, \widetilde{\Theta}^{(v,w,m)}_{L,\varphi,a,b,\alpha}(A) \lesssim  ( \log(b/a))^{1-2^{-n}}.
\end{equation}
 \end{lemma}
Applying Lemma \ref{lemma:ind} specified to  
$(L,\varphi,a,b,\alpha,v,w,m) = (K,g,2^{-J-5}t,2^{-5}t,\widetilde{\beta},1,k_1,m)$ gives the desired estimate for the form \eqref{form:tobound}. Indeed, we obtain a bound for \eqref{mainbound} by
\[ J^{1-2^{-n}} \varepsilon^{-\varrho} \sum_{m=1}^{k_1+1} \int_{(1,\infty)^{k+n}} \int_{\varepsilon}^1 \, \frac{dt}{t} \, \mathcal{D}(\beta)^{-2}  \,d\beta   \]
 In the end it remains to integrate in $t$ and $\beta$ and sum in $m$, which yields the conclusion of Lemma \ref{errlemma}. 
 
\begin{proof}[Proof of Lemma \ref{lemma:ind}] This lemma   will be proven by downward induction on a parameter $0\leq \tau \leq n$.
Let $P(\tau)$ be the statement that the estimates  
\begin{equation*} 
    |\Theta^{(v)}_{L,a,b,\alpha}(A)|, \widetilde{\Theta}^{(v,w,m)}_{L,\varphi,a,b,\alpha}(A) \lesssim ( \log(b/a))^{1-2^{-n+\tau}}
\end{equation*}
 hold for all tuples $L$ satisfying $l_i=1$ and $\alpha_{i}^0=\alpha_{i}^1$ for $1\leq i \leq \tau$. The estimate on \eqref{thetabd}  will follow from $P(0)$. The base case is $P(n)$.   Pick $0\leq \tau \leq n$. We may inductively assume that $P(\tau')$ holds for all $\tau<\tau'\leq n$.

   Let $L$ be a tuple  satisfying $l_i=1$ for $1\leq i \leq \tau$.  Let $1\leq v \leq n$.    First we reduce the bound for \eqref{theta} to the bound for \eqref{thetatilde}. 
Applying the convolution identity  
\begin{equation}-\frac{1}{2} (\Delta g)_{\sqrt{2} s}(x_v^{w}-x_v^{0}) = \sum_{m=1}^{k_{v}+1} \int_{\mathbb{R}^{k_v+1}} (\partial_m g)_s (x_v^{0}-q) (\partial_m g)_s (x_v^{w}-q)\, dq,  \label{khhconv}\end{equation}
rewriting the integral in $x_v^w$ as the innermost,
localizing $(x_1^0,\ldots, x_n^0)$ to the cube $[0,1]^{k+n}$,   estimating all characteristic functions $\mathbbm{1}_A$ which do not depend on the variable  $x_v^w$ by $1$ and using the triangle inequality gives
 \begin{equation}\label{tobd}
     |\Theta^{(v)}_{L,a,b,\alpha}(A)| \leq \sum_{w =1}^{l_{v}} \sum_{m=1}^{k_{v}+1} \widetilde{\Theta}^{(v,w,m)}_{L,|\partial_m g|,a,b,\widetilde{\alpha}}(A),
 \end{equation}
 where $\widetilde{\alpha}_{v}^0  = \widetilde{\alpha}_{v}^w = 2^{-\frac{1}{2}}\alpha_{v}^w$ and $\widetilde{\alpha}$ agrees with $\alpha$ in all other coordinates. 
 Thus, it suffices to bound 
 \begin{equation}
     \label{thetatitde2}
     \sum_{m=1}^{k_v+1}\widetilde{\Theta}^{(v,w,m)}_{L,\varphi,a,b,\alpha}(A)
 \end{equation}
 for each fixed $w$ and any $L,\varphi,a,b,\alpha$ as above.

We  apply the Cauchy-Schwarz inequality in all variables but $x_{v}^{w}$ and   the summation in $m$, which estimates  \eqref{thetatitde2} by the geometric mean of 
\begin{equation}   \sum_{m=1}^{k_{v}+1}  \int_a^b \int_{\mathbb{R}^{(K+1)\cdot (L+1)}}   \Big |  \int_{\R^{k_{v}+1}}    \Big( \prod_{h\in \H_L: h(v)=w}  \mathbbm{1}_A (\Pi_h x)\Big)   (\partial_m g)_{s\alpha_{v}^w} (x_{v}^{w}-q) dx_{v}^{w} \Big |^2  \,  d{\mu}^{(v,w)}_{L,\varphi,\alpha}\, \frac{ds}{s}    \label{cs11} \end{equation}
and
\begin{equation} \sum_{m=1}^{k_{v}+1}  \int_a^b \int_{\R^{(K+1)\cdot (L+1)}}
\mathbbm{1}_{[0,1]^{k+n}}(x_1^0,\ldots, x_n^0)\,  d{\mu}^{(v,w)}_{L,\varphi,\alpha} \, \frac{ds}{s}.  \label{cs21} \end{equation}
Let us focus on the second term \eqref{cs21}, which we estimate trivially.  
 Expanding out the measure,   we first integrate in $q$.  Then for each $1\leq i\leq n$, we integrate   in all $x_i^r$ for $r\neq 0$ that appear in $\mu$ and finally in $x_1^0,\ldots,x_n^0$. 
 In the end we integrate in $s$ and sum in $m$,  which yields a bound for  \eqref{cs21} by an absolute constant times 
 \begin{equation}
 \label{trvbd}
 \log(b/a). 
 \end{equation}

It remains to estimate   \eqref{cs11}.
We integrate in $x_{v}^{r}$ for $r \neq w$ and finally in $x_{v}^0$.  Expanding out also the square of the integrand we see that \eqref{cs11}  is up to a constant  equal to 
\begin{equation*}   \sum_{m=1}^{k_{v}+1}   \int_a^b \int_{\mathbb{R}^{(K+1)\cdot (L+1) -l_v(k_v+1)}}   \Big (  \int_{\R^{k_{v}+1}}    \Big( \prod_{h\in \H_L: h(v)=0}  \mathbbm{1}_A (\Pi_h x)\Big)   (\partial_m g)_{s\alpha_{v}^w} (x_{v}^{0}-q) \, dx_{v}^{0} \Big ) 
 \end{equation*}
 \[ \times    \Big (  \int_{\R^{k_{v}+1}}    \Big( \prod_{h\in \H_L: h(v)=1}  \mathbbm{1}_A (\Pi_h x)\Big)   (\partial_m g)_{s\alpha_{v}^w} (x_{v}^{1}-q) \, dx_{v}^{1} \Big )   \]
 \[\times \Big( \prod_{i=1, i\neq v}^n g_{s\alpha_{i}} ((x_i^r-x_i^0)_{r=1}^{l_i})\,  dx_i
 \Big ) \,  dq \,  \frac{ds}{s} .  \]
  Here we have also relabelled $x_{v}^{w}$ to $x_v^0$ and denoted by $x_v^1$  its copy after the square expansion.

 We can rewrite
 \[ \Big( \prod_{h\in \H_L: h(v)=0}  \mathbbm{1}_A (\Pi_h x)\Big) \Big( \prod_{h\in \H_L: h(v)=1}  \mathbbm{1}_A (\Pi_h x)\Big) = \prod_{h\in \H_{\widetilde{L}}}  \mathbbm{1}_A (\Pi_h x),\]
 where $\widetilde{L}$ agrees with $L$ in all coordinates except possibly at $v$-th, where  $\widetilde{l}_{v} =1$. 
 Using \eqref{khhconv} again we recognize this as a positive constant multiple of
 \begin{equation}\label{theta-2}
     \Theta^{(z)}_{\widetilde{L},a,b,\widetilde{\alpha}}  
 \end{equation}
specified to $z=v$,  where $\widetilde{\alpha}$ agrees with $\alpha$ in all slots except possibly at $v$-th, where  $\widetilde{\alpha}_{v} = (\alpha_{v}^w,\alpha_{v}^w)$.

We distinguish two cases. 
If $v> \tau$, then  by symmetry we may assume $v=\tau+1$.    
The induction hypothesis, i.e. the statement $P(\tau+1)$, gives
\[ \Theta^{(v)}_{\widetilde{L},a,b,\widetilde{\alpha}}
\lesssim (\log(b/a))^{1-2^{-n+\tau+1}}.\]
Taking the geometric mean of this bound and   \eqref{trvbd}  yields the desired estimate for \eqref{tobd}. 

It remains to consider the case $v\leq \tau$, which will also prove the base case. 
If $v\leq \tau$, then   $l_{v}=1$.
First we observe that for any $1\leq z \leq n$ and any $L,a,b,\alpha$ as in \eqref{theta}, 
the form  $\Theta^{(z)}_{{L},a,b,{\alpha}}$ is non-negative if  $l_z=1$ and $\alpha_{z}^0=\alpha_{z}^1$. Indeed,  using \eqref{khhconv} we can write it up to a positive constant   as
\[  \sum_{m=1}^{k_{z}+1}  \int_a^b \int_{\mathbb{R}^{(K+1)\cdot (L+1)}}   \Big ( \int_{\R^{k_{z}+1}}    \Big( \prod_{h\in \H_{L}: h(z)=0}  \mathbbm{1}_A (\Pi_h x)\Big)   (\partial_m g)_{s\alpha_{z}^1} (x_{r}^{0}-q)\,  dx_{r}^{0} \Big )^2    \] 
\[ \times \Big(  \prod_{i=1, i\neq z}^n g_{s\alpha_{i}} ((x_i^r-x_i^0)_{r=1}^{l_i}) \, 
dx_i
\Big)\,  dq \, \frac{ds}{s}. \]
Furthermore, by  the fundamental theorem of calculus in $s$,  the Leibniz rule,  the heat equation \eqref{heateqn},   and the  identity  \eqref{laplaciansplit},  we have for any $L,a,b,\alpha$ as above  
\begin{equation}
    \label{telescoping}
     \sum_{z=1}^n  \Theta^{(z)}_{{L},a,b,{\alpha}}=2\pi \big( \Xi_{L,a,\alpha} - \Xi_{L,b,\alpha} \big), 
\end{equation}
where for any $L,a,\alpha$ as above we have denoted
\[\Xi_{L,a,\alpha}(A) :=   \int_{\mathbb{R}^{(K+1)\cdot (L+1)}}   \Big( \prod_{h\in \H_L} \mathbbm{1}_A(\Pi_hx) \Big)   \prod_{i=1}^n g_{a\alpha_{i}} ((x_i^r-x_i^0)_{r=1}^{l_i} )\, dx. \]
We also have the estimate
\begin{equation}
    \label{xibound}
    \Xi_{L,a,\alpha}(A) \lesssim 1, 
\end{equation}
  can be seen by localizing  $(x_1^0,\ldots,x_n^0)$ to the cube $[0,1]^{k+n}$, bounding  all functions $\mathbbm{1}_A$ by $1$,   integrating first in $x_i^r$ for all $r\neq 0$ and finally in $x_i^0$.

Recall that to finish the proof of the lemma it remains to estimate \eqref{theta-2} for $z=v$ 
where  $v\leq \tau$.  By non-negativity of \eqref{theta-2} for each $1\leq z \leq \tau$, 
 it then suffices to estimate  
\begin{equation*}
\sum_{z=1}^{\tau}\Theta^{(z)}_{\widetilde{L},a,b,\widetilde{\alpha}}. 
\end{equation*}
In order to bound this sum for $1\leq \tau \leq n$, by     \eqref{telescoping} and \eqref{xibound} 
   applied to  $\widetilde{L},a,b,\widetilde{\alpha}$  it    suffices to estimate  the form  \eqref{theta-2} for each $z >  \tau$.  However this  has already been achieved.  
\end{proof}

\section{The uniform part: Proof of lemma \ref{unilemma}}
In this section we proceed similarly as in \cite{LM18} and \cite{Kov20}, with the necessary modifications arising from the multilinear setting. 

We again use the fundamental theorem of calculus, the Leibniz rule and the heat equation \eqref{heateqn}, which yields
\begin{equation}
\label{diffexpr}
\mathcal{N}^{\varepsilon}_{\lambda}(A) - \mathcal{N}^{\vartheta}_{\lambda}(A)  = \frac{1}{2\pi}\sum_{v=1}^n N(v, \lambda,\vartheta, \varepsilon,A),
\end{equation}
where $N(v, \lambda,\vartheta, \varepsilon,A)$ has been defined in \eqref{Nform}. 
We write out each  $N(v, \lambda,\vartheta, \varepsilon,A)$ as in \eqref{Nform-2}. By symmetry it again suffices to 
consider the case $v=1$ and the summand for $w=k_1$. 

In the summand for $w=k_1$ in  $N(1, \lambda,\vartheta, \varepsilon,A)$      we consider the integrand in $p_1^1,\ldots, p_1^{k_1-1}$  and $x_1=(x_1^0,\ldots,x_1^{k_1}),$  i.e.
 \begin{equation}   \label{unif-integrand} \int_{\R^{2k_1(k_1+1)}} \Big( \prod_{h \in \H} \mathbbm{1}_{A}(\Pi_h x) \Big )    
 \sigma_{\lambda}^{\Delta_{1},p_{1}^1,\ldots, p_{1}^{k_1-1}}*(\Delta g)_{t\lambda} (x_{1}^{k_1} - x_{1}^0)     \, g_{t\lambda} \big ((x_{1}^r - x_{1}^0 - p_{1}^r)_{r=1}^{k_1-1} \big ) \,  d\nu\, 
dx_1 
 \end{equation}
where 
\[ d\nu =  d\sigma_{\lambda}^{\Delta_{1},p_{1}^{1},\ldots p_{1}^{k_{1}-2}} (p_{1}^{k_{1}-1})\cdots d\sigma_{\lambda}^{\Delta_{1}}(p_{1}^{1})  . \]
Integrating in $x_1^{k_1}$ we  rewrite it as 
\[ \int_{\R^{(2k_1-1)(k_1+1)}} \Big( \prod_{h \in \H : h(1) \neq k_1} \mathbbm{1}_{A}(\Pi_h x) \Big )  \mathcal{F} \ast {\sigma}_{\lambda}^{\Delta_1,p_1^1,\ldots,p_1^{k_1-1}} \ast (\Delta g)_{t\lambda} (x_1^0)  \]
\begin{equation}
    \label{unfptform}
    \times   g_{t\lambda}  \big ((x_{1}^r - x_{1}^0 - p_{1}^r)_{r=1}^{k_1-1} \big ) \,  d\nu\,  dx_1^0 \cdots dx_1^{k_1-1}, 
\end{equation}
where $\mathcal{F}$ denotes the function
\[ x_1^{k_1} \mapsto \prod_{h \in \H : h(1) = k_1} \mathbbm{1}_{A}(\Pi_h x). \]
By Jensen's inequality  applied to the probabilistic measure 
\[    g_{t\lambda}  \big ((x_{1}^r - x_{1}^0 - p_{1}^r)_{r=1}^{k_1-1} \big ) \mathbbm{1}_{[0,1]^{k_1+1}}(x_1^0)\, d\nu  \,  dx_1^0 \cdots dx_1^{k_1-1}  \]
for fixed $p_1^r$, 
the square of \eqref{unfptform} is bounded by  
\[  \int_{\R^{(2k_1-1)(k_1+1)}} \Big( \prod_{h \in \H : h(1) \neq k_1} \mathbbm{1}_{A}(\Pi_h x) \Big )  |\mathcal{F} \ast {\sigma}_{\lambda}^{\Delta_1,p_1^1,\ldots,p_1^{k_1-1}}  \ast (\Delta g)_{t\lambda} (x_1^0)|^2  \]
\[\times  g_{t\lambda}  \big ((x_{1}^r - x_{1}^0 - p_{1}^r)_{r=1}^{k_1-1} \big ) \, d\mu \, dx_1^0 \cdots dx_1^{k_1-1} . \]
Bounding all characteristic functions that do not make up $\mathcal{F}$ by one and integrating in $x_1^1,\ldots, x_1^{k_1-1}$,  we estimate the   last display by
\[ \int_{\R^{k_1(k_1+1)}}  |\mathcal{F} \ast {\sigma}_{\lambda}^{\Delta_1,p_1^1,\ldots,p_1^{k_1-1}} \ast (\Delta g)_{t\lambda} (x_1^0)|^2  \, dx_1^0  \,d\nu.   \]
By Plancherel's theorem in $x_1^0$, this equals 
\begin{equation}
\label{estFour} 
\int_{\R^{k_1(k_1+1)}}  |\widehat{\mathcal{F}}(\xi)|^2  |\widehat{{\sigma}}_{\lambda}^{\Delta_1,p_1^1,\ldots,p_1^{k_1-1}}(\xi)|^2 |\widehat{(\Delta g)_{t\lambda}}(\xi)|^2  \,  d\xi \,d\nu .
\end{equation}
 We will use the following estimate for the spherical measure that can be found in \cite{Bou86} and \cite{LM18}:  
\[ |\widehat{{\sigma}}^{\Delta_1,p_1^1,\ldots,p_1^{k_1-1}}(\xi) | \lesssim     \textnormal{dist}(\xi, \textnormal{span} \{p_1^1, \ldots, p_1^{k_1-1}\} )^{-\frac{1}{2}} \]
for each $\xi\in \R^{k_1+1}$.
Together with 
\[ \widehat{\Delta g}(\xi) = -4\pi^2 | \xi |^2 e^{-\pi | \xi |^2}, \]
this allows us to estimate \eqref{estFour} by 
\begin{equation}
\label{rotinv}   \int_{\R^{k_1(k_1+1)}} \  |\widehat{\mathcal{F}}(\xi)|^2  \textnormal{dist}(\lambda\xi, \textnormal{span} \{p_1^1, \ldots, p_1^{k_1-1}\})^{-1} | t\lambda\xi |^4 e^{-2\pi | t\lambda\xi |^2}\, d\xi\,   d\nu.  \end{equation}
Next, we have  
\[\int_{\R^{(k_1+1)(k_1-1)}} \textnormal{dist}(\lambda\xi, \textnormal{span} \{p_1^1, \ldots, p_1^{k_1-1}\})^{-1} \, d\nu \lesssim |\lambda\xi|^{-1}\]
This  estimate is   proven in Section 3.3 of \cite{Kov20}. Applying this to \eqref{rotinv} yields a bound by
\[\int_{\R^{k_1+1}} |\widehat{\mathcal{F}}(\xi)|^2  |\lambda \xi|^{-1} |t\lambda \xi|^4 e^{-2\pi |t\lambda \xi|^2} \, d\xi  \lesssim t \int_{\R^{k_1+1}} |\widehat{\mathcal{F}}(\xi)|^2 \, d\xi  \leq t,\]
For the second inequality we used that
$ y\mapsto  y^3 e^{-2\pi y^2}$ is uniformly bounded, while   for the last 
 inequality we applied Plancherel one more time. 
This gives an estimate for  \eqref{unif-integrand}   by  $t^{\frac{1}{2}}.$  Finally,   applying this estimate first to \eqref{unif-integrand} and then integrating in \eqref{diffexpr} in the remaining variables,  we  see that
$$|\mathcal{N}^{\vartheta}_{\lambda}(A)-\mathcal{N}^{\varepsilon}_{\lambda}(A)| \lesssim \int_{\vartheta}^{\varepsilon} \int_{\R^{k+n-k_1-1}} t^{\frac{1}{2}} \mathbbm{1}_{[0,1]^{k+n-k_1-1}}(x_{2}^0,\ldots,x_n^0)\,  dx_{2}^0 \cdots dx_n^0  \frac{dt}{t}\lesssim \varepsilon^{\frac{1}{2}}.$$
Letting $\vartheta \to 0$ yields   \eqref{lemma3}.\\
 
\subsection*{\bf Acknowledgments.}
M. Stip\v{c}i\'{c} was supported in part by a fellowship through the Grand Challenges Initiative at Chapman University and the Croatian Science Foundation under the project UIP-2017-05-4129 (MUNHANAP).
The authors thank Vjekoslav Kova\v{c} for valuable discussions and comments.



\begin{thebibliography}{99}

\bibitem{Bou86}
J. Bourgain, \emph{A Szemer\'{e}di type theorem for sets of positive density in ${\R}^k$.} Israel J. Math. {\bf 54} (1986), no. 3, 307--316.

\bibitem{CMP15}
B. Cook, \'{A}. Magyar, and M. Pramanik, \emph{A Roth type theorem for dense subsets of $\mathbb{R}^d$.} Bull. London Math. Soc., v {\bf 49} (4), pp. 1473--1489

\bibitem{DK20} P. Durcik, V. Kova\v{c}, \emph{A Szemer\'{e}di-type theorem for subsets of the unit cube},  Anal. PDE. {\bf 15} (2022), no. 2, 507-549.

\bibitem{DK18}  P. Durcik, V. Kova\v{c}, \emph{Boxes, extended boxes, and sets of positive upper density in the Euclidean space} (2018), Math. Proc. Cambridge Philos. Soc. {\bf 171} (3), pp. 481--501.

\bibitem{DKT16}
P. Durcik, V. Kova\v{c}, and C. Thiele, \emph{Power-type cancellation for the simplex Hilbert transform}, J. Anal. Math. {\bf 139} (2019), 67--82.

\bibitem{DST}
P. Durcik, L. Slav{\'i}kov{\'a}, and C. Thiele, \emph{Local bounds for singular Brascamp-Lieb forms with cubical structure}, preprint (2021), arXiv:2112.15101.


\bibitem{FM}
K. J. Falconer, J. M. Marstrand, \emph{Plane sets with positive density at infinity contain all large distances}, Bull. London Math. Soc. {\bf 18} (1986), no. 5, 471--474.

\bibitem{FKW}
H. Furstenberg, Y. Katznelson, B. Weiss, \emph{Ergodic theory and configurations in sets of positive
density. Mathematics of Ramsey theory}, pp. 184--198, Algorithms Combin. 5, Springer, Berlin,
1990.

 
\bibitem{Kov11}
	V. Kova\v{c}, \emph{Bellman function technique for multilinear estimates and an application to generalized paraproducts}. Indiana Univ. Math. J. {\bf 60} (2011), no. 3, 813--846.

\bibitem{Kov20}
V. Kova\v{c}, \emph{Density theorems for anisotropic point configurations}. To appear in Canad. J. Math. (2020). 


\bibitem{LT}
M. Lacey, C. Thiele, \emph{Lp estimates for the bilinear Hilbert transform.} Proc. Nat. Acad. Sci. U.S.A.,
94 (1997), no. 1, 33--35.

 \bibitem{LM20}
 N. Lyall and A. Magyar, \emph{Distance graphs and sets of positive upper density in $\R^d$.} Analysis \& PDE 13 (2020) 685--700.


 \bibitem{LM18}
 N. Lyall and A. Magyar, \emph{Product of simplices and sets of positive upper density in $\R^d$.} Math. Proc. Cambridge
  Philos. Soc. {\bf 165} (2018), no. 1, 25--51.
  
  
\bibitem{LM19}
N. Lyall and A. Magyar, \emph{Weak hypergraph regularity and applications to geometric Ramsey theory.} To appear in Trans. Amer. Math. Soc.  (2019).


 \bibitem{Sti19}
	M. Stip\v{c}i{\'c}, \emph{$T(1)$ theorem for dyadic singular integral forms associated with hypergraphs.} J. Math. Anal. Appl. {\bf 481} (2020), no. 2, Article 123496.


\bibitem{Sz83}
Laszlo Sz{\'e}kely. \emph{Remarks on the chromatic number of geometric graphs.} In Graphs and other combinatorial topics (Prague, 1982), volume 59 of Teubner-Texte Math., pages 312–315. Teubner, Leipzig, 1983.

\bibitem{Tao15}
T. Tao, \emph{Cancellation for the multilinear Hilbert transform.} Collect. Math. {\bf 67} (2016), no. 2, 191--206.


\bibitem{Zor15}
P. Zorin-Kranich, \emph{Cancellation for the simplex Hilbert transform.} Math. Res. Lett. {\bf 24} (2) (2017), pp. 581--592. 
\end{thebibliography}
\end{document}